\documentclass{amsart}
\usepackage[latin1]{inputenc}
\usepackage{amsmath,amssymb,graphicx,setspace,verbatim}
\usepackage[pagewise]{lineno}
\usepackage{color}
\usepackage{appendix}
\theoremstyle{comment}
\usepackage[color,matrix,arrow]{xy}
\usepackage{cleveref}
\newtheorem*{mcomment}{\color{cyan}{Comment}}

\usepackage{multicol}
\usepackage{multirow}
\usepackage{url}
\input xy
\xyoption{all}
\newcommand*{\nfrac}[2]{\genfrac{}{}{0pt}{}{#1}{#2}}

\newtheorem{theorem}{Theorem}[section]

\newtheorem{lemma}[theorem]{Lemma}
\newtheorem{corollary}[theorem]{Corollary}
\newtheorem{proposition}[theorem]{Proposition}
\theoremstyle{definition}
\newtheorem{definition}[theorem]{Definition}

\title{Infinite Families of Hypertopes from Centrally Symmetric Polytopes}

\author[Claudio Alexandre Piedade]{Claudio Alexandre Piedade}
\address{
Center for Research and Development in Mathematics and Applications, Department of Mathematics, University of Aveiro, Portugal
}
\email{claudio.a.piedade@ua.pt}

\begin{document}

\begin{abstract}
We construct infinite families of abstract regular polytopes of type $\{4,p_1,\ldots,p_{n-1}\}$ from extensions of centrally symmetric spherical abstract regular $n$-polytopes. In addition, by applying the halving operation, we obtain infinite families of both locally spherical and locally toroidal regular hypertopes of type $\left\{\nfrac{p_1}{p_1},\ldots,p_{n-1}\right\}$.
\end{abstract}

\maketitle


\noindent \textbf{2010 Math Subj. Class:} 51E24, 52B11, 20F05

\section{Introduction}
Polytope theory is a well studied area of algebra and geometry. The concept of an abstract polytope was introduced in \cite{ARP} as a poset whose elements are faces.
Another way of defining an abstract polytope is as an incidence geometry which is thin, residually-connected, flag-transitive and has a linear Coxeter diagram. 
This idea of seeing polytopes as incidence geometries was recently used in \cite{HSH}, in which the authors decided to generalize this concept, dropping the linear diagram condition and naming these new structures hypertopes. 

Since the introduction of the concept of a hypertope, many examples have been given \cite{catalanohypertopes,exitchamb}, especially locally toroidal families \cite{rank4toroidal,hexagonal,FLPW2019}.
In \cite{spherical}, the authors classified the locally spherical regular hypertopes of spherical, euclidean or hyperbolic type, giving new examples of the hyperbolic ones using MAGMA. 
During the 9th Slovenian International Conference on Graph Theory in Bled (2019), Asia Ivi\'{c} Weiss posed me the problem of extending these examples into infinite families of locally spherical regular hypertopes of hyperbolic type. Since then, in \cite{teroasia2020}, Weiss and Montero have described two proper locally spherical hypertopes of hyperbolic type for ranks 4 and 5. More recently \cite{teroasia2021}, the same authors have given a constructive way to obtain families of locally spherical hypertopes of hyperbolic type using Schreier coset graphs.

The halving operation is quite common in regular polytopes of type $\{4,p\}$ ($p\geq 3$), where we double the fundamental region of the polytope, resulting in a self-dual polytope of type $\{p,p\}$. This operation was broadened to any non-degenerate abstract regular $n$-polytope by Weiss and Montero in \cite{teroasia2020}, in which its use on Danzer's polytopes $2^{\mathcal{P}}$ \cite{danzer_regular_1984} results in the two examples of locally spherical hypertopes of hyperbolic type mentioned above.

Here, we will obtain new families of regular hypertopes, some of which are locally spherical hypertopes of hyperbolic and euclidean type, while others are locally toroidal. To do so, we will extend centrally symmetric polytopes into a family of polytopes denoted as $2^{\mathcal{P},\mathcal{G}(s)}$ with type $\{4,p_1,\ldots,p_{n-1}\}$, where we can obtain regular hypertopes through the halving operation. 
The extension of polytopes is not new in polytope theory \cite{pellicer_extensions_2009} and the construction of the $2^{\mathcal{P},\mathcal{G}(s)}$ polytopes is well described in \cite[Section 8]{ARP}. Our families of extended polytopes $2^{\mathcal{P},\mathcal{G}(s)}$, when we consider $s=2$, give duals of the examples of locally spherical hypertopes of hyperbolic type given by Weiss and Montero in \cite{teroasia2020}. Moreover, the automorphism group presentation and its order are given for each of the polytopes $2^{\mathcal{P},\mathcal{G}(s)}$ and respective hypertopes, allowing this work to be extended further, as discussed in Section~\ref{futurework}.

In Section~\ref{back}, an introduction to abstract regular polytopes, regular hypertopes and the halving operation is presented. In Section~\ref{twisting}, we introduce the polytopes $2^{\mathcal{P},\mathcal{G}(s)}$. In Sections~\ref{rank2&n} and \ref{rank3&4} the families of polytopes and hypertopes are given, as well as their automorphism group presentation and order. Specifically, in Sections~\ref{ss:35} and \ref{ss:335} we give two families of locally spherical regular hypertopes of hyperbolic type of rank 4 and 5, respectively. Moreover, in Section~\ref{northoplex-ext} a family of arbitrary rank of locally spherical hypertopes of euclidean type is given and in Sections~\ref{ncubehyper} and \ref{ss:343} two distinct families of hypertopes of rank 4 and 5 are given from the halving of quotients of locally toroidal polytopes $\{4,4,3\}$ and $\{4,3,4,3\}$, respectively.

\section{Background}\label{back}

\subsection{C-groups}

A \emph{C-group of rank $n$} is a pair $(G,S)$, where $G$ is a group and $S:=\{\rho_0,\ldots,\rho_{n-1}\}$ is a generating set of involutions of $G$ that satisfy the \emph{intersection property}, i.e. $$\forall I,J\subseteq \{0,\ldots,n-1\},\ \langle \rho_i\ | \ i\in I\rangle\ \cap\ \langle \rho_j\ | \ j\in J\rangle\ = \ \langle \rho_k\ |\ k \in I\cap J\rangle.$$
A C-group is said to be a \emph{string C-group} if its generating involutions can be ordered in such a way that, for all $i,j$ where $|i-j|\geq 2$, $(\rho_i\rho_j)^2 = id$.
A subgroup of $G$ generated by all but one involution of $S$ is called a \emph{maximal parabolic subgroup} and is denoted as $$G_i := \langle \rho_j\ |\ j\in I\backslash\{i\}\rangle,$$ with $I:=\{0,\ldots, n-1\}$.

The \emph{Coxeter diagram} of a C-group $(G,S)$ is the graph whose nodes represent the elements of $S$ and an edge between the generators $i$ and $j$ has label $p_{ij}:=o(\rho_i\rho_j)$, the order of $\rho_i\rho_j$. By convention, edges with label equal to 2 are improper and are not drawn, and whenever an edge has label 3, its label is omitted. For string C-groups, this diagram is linear, like the one represented in Figure~\ref{coxeterdiagram}.
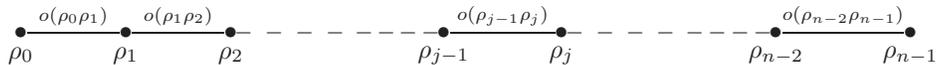
\begin{figure}[h]
 $$\xymatrix@-1.9pc{
*{\bullet} \ar@{-}[rrrrr]^{o(\rho_0\rho_1)} && & &&*{\bullet} \ar@{-}[rrrrr]^{o(\rho_1\rho_2)} && & &&*{\bullet} \ar@{--}[rrrrrr] && && && \ar@{--}[rrrr] && && *{\bullet} \ar@{-}[rrrrr]^{o(\rho_{j-1}\rho_j)} && & &&  *{\bullet} \ar@{--}[rrrrrr] && && && \ar@{--}[rrrr] && && *{\bullet} \ar@{-}[rrrrr]^{o(\rho_{n-2}\rho_{n-1})} && & && *{\bullet} \\
*{\rho_0} && & &&*{\rho_1} && & && *{\rho_2} && && && && && *{\rho_{j-1}} && & && *{\rho_{j}} && && && && && *{\rho_{n-2}} && & && *{\rho_{n-1}}\\
}$$
 \caption{Coxeter Diagram of a string C-group.}
 \label{coxeterdiagram}
\end{figure}
We say a C-group is a \emph{Coxeter group} when its group relations are just the ones given by its Coxeter diagram.
As we shall see, the automorphism groups of both abstract regular polytopes and regular hypertopes are C-groups.

\subsection{Regular Hypertopes and Regular Polytopes}\label{hyp}

The term hypertope was introduced in \cite{HSH} as a generalization of polytopes whose automorphism group is a C-group, but not necessarily a string C-group. A \emph{regular hypertope} is an incidence geometry that is residually-connected, thin and flag-transitive, as defined in \cite{HSH}. Here, we will focus on its construction from C-groups.

An \emph{incidence system} is a 4-tuple $\Gamma := (X,\ast, t, I)$ satisfying the following conditions:
\begin{itemize}
 \item $X$ is a set with the \emph{elements} of $\Gamma$;
 \item $I$ is a set with the \emph{types} of $\Gamma$;
 \item $t : X\rightarrow I$ is a \emph{type function}, attributing to each element $x\in X$ a type $t(x)\in I$; we call $x$ an \emph{$i$-element} if $t(x)=i$, for $i\in I$; and
 \item $\ast$ is a binary relation in $X$ called \emph{incidence}, which is reflexive, symmetric and such that, for all $x,y\in X$, if $x\ast y$ and $t(x)=t(y)$, then $x=y$.
\end{itemize}
The cardinality of $I$ is called the \emph{rank} of $\Gamma$.
A \emph{flag} is a set of pairwise incident element of $\Gamma$. For a flag $F$, the set $t(F):=\{t(x)|x\in F\}$ is called the \emph{type of F}, and we say $F$ is a \emph{chamber} when $t(F)=I$. An incidence system $\Gamma$ is called a \emph{geometry} or \emph{incidence geometry} if every flag of $\Gamma$ is a subset of a chamber. Consider the following proposition.

\begin{proposition}[Tits Algorithm, \cite{Tits1961}]\label{prop1}
 Let $n$ be a positive integer and $I:=\{0,\ldots,n-1\}$. Let $G$ be a group together with a family of subgroups $(G_i)_{i\in I}$, $X$
 the set consisting of all cosets $G_ig$ with $g\in G$ and $i\in I$, and $t:X\rightarrow I$ defined
 by $t(G_ig)=i$. Define an incidence relation $\ast$ on $X\times X$ by: $$G_ig_1\ast G_jg_2 \textnormal{ if and only if } G_ig_1\cap G_jg_2 \neq \emptyset .$$
 Then the 4-tuple $\Gamma:=(X,\ast, t,I)$ is an incidence system having a chamber.
 Moreover, the group $G$ acts by right multiplication as an automorphism group on $\Gamma$.
 Finally, the group $G$ is transitive on the flags of rank less than 3.
\end{proposition}

The incidence system constructed using the proposition above will be denoted by $\Gamma(G; (G_i)_{i\in I})$ and designated as a \emph{coset geometry} if the incidence system is a geometry. The subgroups $(G_i)_{i\in I}$ are the maximal parabolic subgroups of $G$, stabilizers of an element of type $i$. 
Let $G$ be a C-group. We say $G$ is \emph{flag-transitive} on $\Gamma(G; (G_i)_{i\in I})$ if $G$ is transitive on all flags of a given type $J$, for each type $J\subseteq I$. We have the following result, which defines a regular hypertope from a C-group.

\begin{theorem}\cite[Theorem 4.6]{HSH}\label{theorem46}
Let $G=\langle \rho_i\,|\,i\in I\rangle$ be a C-group and let $\Gamma := \Gamma(G;(G_i)_{i\in I})$ where $G_i:=\langle \rho_j\,|\,j\neq i\rangle$ for all $i\in I$.
If $G$ is flag-transitive on $\Gamma$, then $\Gamma$ is a regular hypertope.
\end{theorem}

A regular hypertope with a string Coxeter diagram is an \emph{(abstract) regular polytope}; conversely, if $G$ is a string C-group, the coset incidence system $\Gamma(G;(G_i)_{i\in I})$ is a regular hypertope with string diagram \cite[Theorem 5.1 and Theorem 5.2]{HSH}. Thus, regular polytopes are particular cases of regular hypertopes. We say a regular hypertope is \emph{proper} if it does not have a linear Coxeter diagram. Moreover, regular polytopes are in one-to-one correspondence with string C-groups.

Let $G$ be a string C-group and let $\Gamma(G;(G_i)_{i\in I})$ be the regular polytope built using Proposition~\ref{prop1}. Let $G_{-1} := G =: G_n$. Then, since the generators of a string C-group have a prescribed order, the incidence relation of Proposition~\ref{prop1} is a partial order, where $G_j\phi \leq G_k\psi$ if and only if $-1\leq j\leq k\leq n$ and $G_j\phi \cap G_k\psi\neq \emptyset$. This construction results in a ranked partial ordered set (\emph{poset}) $\mathcal{P}$ which satisfies all axioms of the definition of an abstract regular polytope given in \cite{ARP}. From now on, we will consider abstract regular polytopes as posets instead of coset geometries.

The elements of $\mathcal{P}$ are called \emph{faces} and the rank of the faces are induced by the labeling of the generators of the string C-group. A face $F\in \mathcal{P}$ of rank($F$)$=i$ is called an $i$-face, where the $0$-faces of $\mathcal{P}$ are called \emph{vertices}, the $1$-faces \emph{edges} and the $(n-1)$-faces \emph{facets}. Let $F_0$ be a vertex of $\mathcal{P}$ stabilized by $G_0 := \langle \rho_j | j\neq 0\rangle$. We can identify the vertices of $\mathcal{P}$ by the right cosets of $G_0$.
We say a poset is a \emph{lattice} if, for every two faces $F,G\in\mathcal{P}$, there is a least upper bound and a greatest lower bound for $\{F,G\}$.
Whenever the partial order induces a lattice, we will call $\mathcal{P}$ \emph{non-degenerate}, otherwise we will call it \emph{degenerate} \cite{schulte_regular_1985,danzer_regular_1984}.

The \emph{Schl\"{a}fli type} of a regular polytope $\mathcal{P}$ is defined as $\{p_1,\ldots,p_{n-1}\}$, where $p_i$ is the order of two consecutive generators $o(\rho_{i-1}\rho_i)$.
If a $n$-polytope has type $\{p_1,\ldots,p_{n-1}\}$, we can write its universal automorphism group, a Coxeter group, as $[p_1,\ldots,p_{n-1}]$. The dual polytope of $\mathcal{P}$ is obtained by reversing the partial order of the poset, which is equivalent to reversing the order of the generators of the string C-group. 

Let $\mathcal{P}$ be an abstract regular $n$-polytope and $F_A,F_B\in \mathcal{P}$ be distinct vertices of $\mathcal{P}$.
The unordered pair of vertices $\{F_A,F_B\}$ is called a \emph{diagonal} of $\mathcal{P}$.
Two diagonals $\{F_A, F_B\}$ and $\{F_C,F_D\}$, with $F_C,F_D\in \mathcal{P}$, are said to be \emph{equivalent} if there is some $\sigma\in G(\mathcal{P})$ such that $\{F_C, F_D\} = \{F_A\sigma,F_B\sigma\}$.
Thus, the diagonals of $\mathcal{P}$ form equivalence classes, called \emph{diagonal classes}. 
Since these vertices can be represented as right cosets of $G_0$, we can write a diagonal as $\{G_0\phi, G_0 \psi\}$, where $G_0\phi=F_0\phi = F_A$, $G_0\psi=F_0\psi = F_B$, $F_A\neq F_B$ and $\phi,\psi\in G(\mathcal{P})$. Moreover, due to the transitivity of $G(\mathcal{P})$, we can think of the diagonal classes by their representative $\{G_0, G_0\sigma\}$ for $\sigma\notin G_0$, where we fix one of the vertices as $G_0$. In this case, two diagonals $\{F_0, F_A\} = \{G_0, G_0\phi\}$ and  $\{F_0, F_B\} = \{G_0, G_0\psi\}$ are equivalent under $G(\mathcal{P})$ if and only if 
\begin{equation}\label{eq:doublecosetdiagonal}
 \psi \in G_0\phi G_0 \cup G_0\phi^{-1} G_0,
\end{equation}
for $\phi,\psi\notin G_0$.
If the polytope is realizable in an Euclidean space, the diagonal classes can be ordered by the distance between their representative vertices. For instance, the edges of the polytope $\mathcal{P}$ form a diagonal class.

An abstract regular polytope $\mathcal{P}$ is said to be \emph{centrally symmetric} if its automorphism group $G(\mathcal{P})$ has a proper central involution $\alpha$ which is fixed-point free on its vertices. A pair of vertices of a centrally symmetric polytope is \emph{antipodal} if they are permuted by this central involution. 
In the diagonal classes of a centrally symmetric polytope, there will be a diagonal class of all the pairs of antipodal points, with representative $\{G_0,G_0\alpha\}$.

A regular hypertope is \emph{spherical} if its Coxeter diagram is a union of diagrams of finite irreducible Coxeter groups. 
Moreover, a \emph{locally spherical regular hypertope} is a hypertope whose maximal parabolic residues are spherical hypertopes. 
We say a locally spherical regular hypertope is of \emph{euclidean type} if its Coxeter diagram correspondes to an infinite irreducible Coxeter group of Euclidean type \cite[Table 3B2]{ARP}. A \emph{regular toroidal hypertope} is a quotient of a regular universal hypertope of euclidean type by a normal subgroup of its translational symmetries \cite{spherical}.
A locally spherical regular hypertope is of \emph{hyperbolic type} if its type-preserving automorphism group of its universal cover is of an irreducible compact hyperbolic Coxeter group \cite[Table 2]{spherical}.
Lastly, a regular $n$-polytope or a regular $4$-hypertope is said to be \emph{locally toroidal} if its maximal residues are either spherical or toroidal, with at least one of them being toroidal \cite{ARP,FLPW2019}. A generalization of this concept for hypertopes of rank greater than 4 is yet to be established.

In Table~\ref{tab:centralinvo}, we give a list of all centrally symmetric regular non-degenerate polytopes of spherical type (i.e. finite irreducible Coxeter groups with linear diagram) whose automorphism group is $\langle \tau_0,\ldots,\tau_{n-1}\rangle$ and having central involution $\alpha$.

\begin{table}[h!]
\centering
\caption{The centrally symmetric non-degenerate regular polytopes of spherical type}
\label{tab:centralinvo}
\begin{tabular}{|l|l|c|c|c|}
\hline
Rank                      & \multicolumn{1}{c|}{Schl\"{a}fli type}& Number of vertices  & $\alpha$                                            & $|\mathcal{P}|$         \\ \hline
2                         & $\{2p\}$, for $2\leq p < \infty$         & $2p$  & $(\tau_0\tau_1)^p$                                & $4p$                      \\ \hline
\multirow{4}{*}{3}        & $\{3,4\}$                        & 6   & \multirow{2}{*}{$(\tau_0\tau_1\tau_2)^3$}           & \multirow{2}{*}{48}     \\ \cline{2-3}
                          & $\{4,3\}$              &   8           &                                                   &                         \\ \cline{2-5} 
                          & $\{3,5\}$                &    12        & \multirow{2}{*}{$(\tau_0\tau_1\tau_2)^5$}           & \multirow{2}{*}{120}    \\ \cline{2-3}
                          & $\{5,3\}$                 &    20       &                                                   &                         \\ \hline
\multirow{5}{*}{4}        & $\{3,3,4\}$&      8                    & \multirow{2}{*}{$(\tau_0\tau_1\tau_2\tau_3)^4$}     & \multirow{2}{*}{384}    \\ \cline{2-3}
                          & $\{4,3,3\}$              &    16        &                                                   &                         \\ \cline{2-5} 
                          & $\{3,4,3\}$              &    24        & $(\tau_0\tau_1\tau_2\tau_3)^6$                      & 1152                    \\ \cline{2-5} 
                          & $\{3,3,5\}$              &    120        & \multirow{2}{*}{$(\tau_0\tau_1\tau_2\tau_3)^{15}$}    & \multirow{2}{*}{14400}  \\ \cline{2-3}
                          & $\{5,3,3\}$              &  600          &                                                   &                         \\ \hline
\multirow{2}{*}{$n \geq 5$} & $\{3^{n-2},4\}$    &      $2n$              & \multirow{2}{*}{$(\tau_0\tau_1\ldots\tau_{n-1})^n$} & \multirow{2}{*}{$2^n n!$} \\ \cline{2-3}
                          & $\{4,3^{n-2}\}$       &   $2^n$              &                                                   &                         \\ \hline
\end{tabular}
\end{table}

The only centrally symmetric polygons have an even number of vertices and their proper central involution is a $180$ degrees rotation. For rank 3 and 4, the spherical regular polytopes that are centrally symmetric can be easily computed. For rank $n\geq 5$, the only spherical regular polytopes are the $n$-simplex and the $n$-cube $\{4,3^{n-2}\}$ (and its dual). Since the group of the $n$-simplex is centerless, only the $n$-cube $\{4,3^{n-2}\}$ (and its dual) are centrally symmetric, with $\alpha = (\tau_0\tau_1\ldots\tau_{n-1})^n$ \cite{hartley_new_2008}. Moreover, all the polytopes of Table~\ref{tab:centralinvo} are convex polytopes, meaning that their poset form a face-lattice.

\subsection{Halving Operation of non-degenerate polytopes}\label{halv}

Recently, in \cite{teroasia2020}, the halving operation was revisited and, furthermore, the conditions under which it gives a regular hypertope from a regular polytope were established. In what follows we recall important results that can be found in \cite{teroasia2020}.  

Let $n\geq 3$ and let $\mathcal{P}$ be a regular non-degenerate $n$-polytope of type $\{p_1,\ldots,p_{n-1}\}$ with automorphism group $G(\mathcal{P}) = \langle \rho_0,\ldots,\rho_{n-1}\rangle$. 

The \emph{halving operation} is the map 
$$\eta : \langle \rho_0,\rho_1,\rho_2,\ldots,\rho_{n-1}\rangle \rightarrow \langle \rho_0\rho_1\rho_0,\rho_1,\rho_2,\ldots,\rho_{n-1}\rangle = \langle \tilde{\rho_0},\rho_1,\rho_2,\ldots,\rho_{n-1}\rangle.$$

The \emph{halving group} of $\mathcal{P}$, denoted by $H(\mathcal{P})$, is the image of the halving operation on $G(\mathcal{P})$.
If $n\geq 3$ and $\mathcal{P}$ is non-degenerate, then $H(\mathcal{P}) = \langle \tilde{\rho_0}, \rho_1,\ldots,\rho_{n-1}\rangle$ is a C-group \cite[Theorem 3.1]{teroasia2020} with the following Coxeter diagram

$$
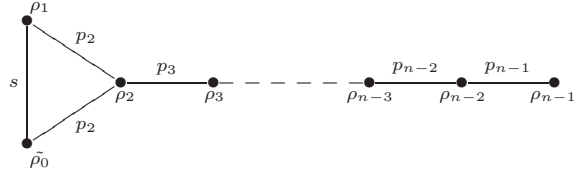
\xymatrix@-1.5pc{*{\bullet}\ar@{-}[rrrdd]^(0.01){\rho_1}^{p_2} \\
\\
 &&&*{\bullet}\ar@{-}[rrr]_(0.05){\rho_2}^{p_3}&&&*{\bullet}\ar@{--}[rrrrr]_(0.01){\rho_3}&&&&&*{\bullet}\ar@{-}[rrr]_(0.01){\rho_{n-3}}^{p_{n-2}}&&&*{\bullet}\ar@{-}[rrr]_(0.01){\rho_{n-2}}_(0.99){\rho_{n-1}}^{p_{n-1}}&&&*{\bullet}\\
 \\
 *{\bullet}\ar@{-}[rrruu]_(0.01){\tilde{\rho_0}}_{p_2}\ar@{-}[uuuu]^s\\
 } $$
where $s=p_1$ if $p_1$ is odd, otherwise $s=\frac{p_1}{2}$. In this paper, we will focus on polytopes of type $\{4,p_2,\ldots,p_{n-1}\}$, hence the Coxeter diagram will be as follows.
$$
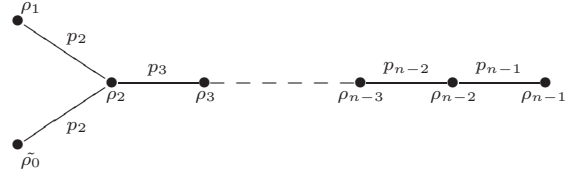
\xymatrix@-1.5pc{*{\bullet}\ar@{-}[rrrdd]^(0.01){\rho_1}^{p_2} \\
\\
 &&&*{\bullet}\ar@{-}[rrr]_(0.05){\rho_2}^{p_3}&&&*{\bullet}\ar@{--}[rrrrr]_(0.01){\rho_3}&&&&&*{\bullet}\ar@{-}[rrr]_(0.01){\rho_{n-3}}^{p_{n-2}}&&&*{\bullet}\ar@{-}[rrr]_(0.01){\rho_{n-2}}_(0.99){\rho_{n-1}}^{p_{n-1}}&&&*{\bullet}\\
 \\
 *{\bullet}\ar@{-}[rrruu]_(0.01){\tilde{\rho_0}}_{p_2}\\
 } $$

Let $\mathcal{H}(\mathcal{P})$ denote the coset indicence system $\Gamma(H(\mathcal{P}), (H_i)_{i\in\{0,\ldots,n-1\}})$ associated with the non-generate regular polytope $\mathcal{P}$, where $H_i$ are the maximal parabolic subgroups of $H(\mathcal{P})$. Then, $\mathcal{H}(\mathcal{P})$ is flag-transitive \cite[Proposition 3.2]{teroasia2020}. Hence, using Theorem~\ref{theorem46}, we have the following corollary.
\begin{corollary}\cite[Corollary 3.2]{teroasia2020}\label{H(P)hypertopes}
 Let $\mathcal{P}$ be a non-degenerate regular $n$-polytope and $I = \{0,\ldots,n-1\}$. Let $H(\mathcal{P})$ be the halving group of $\mathcal{P}$. Then the incidence system $\mathcal{H}(\mathcal{P}) = \Gamma(H(\mathcal{P}), (H_i)_{i\in I})$ is a regular hypertope such that $Aut_I(\mathcal{H}(\mathcal{P})) = H(\mathcal{P})$.
\end{corollary}

If $\mathcal{P}$ is a polytope of type $\{4,p_2,p_3\ldots,p_{n-1}\}$, then we denote the type of $\mathcal{H}(\mathcal{P})$ as $\left\{\nfrac{p_2}{p_2},p_3,\ldots,p_{n-1}\right\}$, and its universal automorphism group as $\left[\nfrac{p_2}{p_2},p_3,\ldots,p_{n-1}\right]$.
Moreover, as stated in \cite{teroasia2020}, $H(\mathcal{P})$ has index $2$ on $G(\mathcal{P})$ if and only if $p_1$ is even.
In Sections~\ref{rank2&n} and \ref{rank3&4} the halving operation will be used on polytopes of type $\{4,p_2,\ldots,p_{n-1}\}$, hence, in all cases that we deal with, $|H(\mathcal{P})| = |G(\mathcal{P})|/2$.

\subsection{The $2^{\mathcal{P},\mathcal{G}(s)}$ polytopes}\label{twisting}

Consider a Coxeter group $W$ generated by $k$ involutions $\sigma_0,\ldots,\sigma_{k-1}$ with Coxeter diagram $\mathcal{G}$, and $\tau_0,\ldots,\tau_{n-1}$ to be involutory automorphisms of $W$, permuting its generators. Then $W$ can be augmented by the group $\Lambda$ generated by these permutations using a semidirect product, resulting in a group $G = W\rtimes \Lambda$. In the case $W$ is a C-group represented by a Coxeter diagram $\mathcal{G}$, the automorphisms $\tau_i$ can be seen as symmetries of $\mathcal{G}$.

\begin{definition}\label{kadmiss}\cite{ARP}
 Let $\mathcal{G}$ be the Coxeter diagram of a C-group and let $\mathcal{P}$ be a regular $n$-polytope with automorphism group $G(\mathcal{P})=\langle \tau_0,\ldots,\tau_{n-1}\rangle$. We say $\mathcal{G}$ is $\mathcal{P}$-\emph{admissible} if:
 \begin{itemize}
  \item The Coxeter diagram $\mathcal{G}$ has more than one node;
  \item $G(\mathcal{P})$ acts transitively on the set of nodes of $\mathcal{G}$, $V(\mathcal{G})$;
  \item The subgroup $\langle \tau_1,\ldots,\tau_{n-1}\rangle$ of $G(\mathcal{P})$ fixes at least one node of $\mathcal{G}$, which we will designate as $F_0$;
  \item The action of $G(\mathcal{P})$ on the diagram $\mathcal{G}$, with respect to $F_0$, respects the intersection property, i.e., for $I\subseteq \{0,\ldots,n-1\}$ and denoting $V(\mathcal{G},I)$ as the set of nodes of $\mathcal{G}$ that the subgroup $\langle \tau_i | i\in I\rangle$ maps the node $F_0$ to, then $$V(\mathcal{G},I)\cap V(\mathcal{G},J) = V(\mathcal{G},I\cap J)\textnormal{ if }I,J\subseteq \{0,\ldots,n-1\}$$
 \end{itemize}
\end{definition}

Let us consider the case $V(\mathcal{G})=V(\mathcal{P})$ and let $F_0$ be a vertex of $\mathcal{P}$. Then $\mathcal{G}$ is $\mathcal{P}$-admissible \cite{ARP}. The number of possible choices for the proper branches of the diagram $\mathcal{G}$ depends on the number of diagonal classes of $\mathcal{P}$. When $\mathcal{P}$ is centrally symmetric, there is an involution $\alpha$ permuting pairs of antipodal vertices of $\mathcal{P}$, forming one diagonal class of $\mathcal{P}$. When this diagonal class is the only one represented in the Coxeter diagram $\mathcal{G}$ by proper branches all with the same label $s$, then the diagram $\mathcal{G}=:\mathcal{G}(s)$ is a matching and the corresponding Coxeter group is a direct product of dihedral groups of degree $s$, $D_s$, which is finite if and only if $\mathcal{P}$ is finite.

Consider a Coxeter group with diagram $\mathcal{G}(s)$, 
\begin{equation}\label{eq:W}
 W := W(\mathcal{G}(s)) = \langle \sigma_F\ |\ F\in V(\mathcal{G}(s)) \rangle,
\end{equation}
 and a centrally symmetric regular polytope $\mathcal{P}$ where $V(\mathcal{G}(s))=V(\mathcal{P})$. Let $F_0$ be a vertex of $\mathcal{P}$. Then the $(n+1)$-polytope $2^{\mathcal{P},\mathcal{G}(s)}$ is defined by the group $$G(2^{\mathcal{P},\mathcal{G}(s)}) := W \rtimes G(\mathcal{P}) = \langle\rho_0,\ldots,\rho_n\rangle $$
 where
\begin{equation}\label{eq:rho_def}
\rho_i := \bigg\{  \begin{matrix} \sigma_{F_0}, & \mbox{for }i = 0, \\ \tau_{i-1}, & \mbox{for }i = 1,\ldots,n+1. \end{matrix}
\end{equation}
For each $F\in V(\mathcal{G})$, there exists an element $\tau \in G(\mathcal{P})$ and an involution $\sigma_F$ of $W$ such that 
\begin{equation}\label{eq:sigmaFtoF0}
 \sigma_F = \sigma_{F_0\tau} = \tau^{-1}\sigma_{F_0}\tau = \tau^{-1}\rho_{0}\tau.
\end{equation}
When $\mathcal{P}$ is centrally symmetric, with central involution $\alpha$, and $\mathcal{G}(s)$ is a matching as before with label $s$, then $(\sigma_F \sigma_{F\alpha})^s = id$. When $s\geq 3$, as all generators $\sigma_{F}$ of $W$ commute with each other, except with $\sigma_{F\alpha}$, $W \cong D_s^{|V(\mathcal{G})|/2}$. Also, for each $F\in V(\mathcal{G})$ and $\tau \in G(\mathcal{P})$, such that $F_0\tau = F$, we have that
\begin{equation}\label{eq:conj_sigmas}
\sigma_F \sigma_{F\alpha} = \tau^{-1}\rho_0\tau \alpha^{-1}\tau^{-1} \rho_0 \tau\alpha = \tau^{-1}\rho_0\alpha\rho_0\alpha\tau.
\end{equation}
In particular, for $s=2$, the diagram $\mathcal{G}(2)$ only has improper branches and $2^{\mathcal{P},\mathcal{G}(2)}$ is the Danzer polytope $2^{\mathcal{P}}$ \cite{danzer_regular_1984}.

The following theorem gives some properties of $2^{\mathcal{P},\mathcal{G}(s)}$ which will be of great importance for our results.

\begin{theorem}\label{2KG(s)}[Theorem 8C5 of \cite{ARP}]
 Let $n\geq 1$, and let $\mathcal{P}$ be a centrally symmetric regular $n$-polytope of type $\{p_1,\ldots,p_{n-1}\}$ with $p_1\geq 3$. Then the regular $(n+1)$-polytope $2^{\mathcal{P},\mathcal{G}(s)}$ has the following properties.
 \begin{enumerate}
  \item $2^{\mathcal{P},\mathcal{G}(s)}$ is of type $\{4,p_1,\ldots,p_{n-1}\}$;
  \item $G(2^{\mathcal{P},\mathcal{G}(s)})= D_s^q\rtimes G(\mathcal{P})$, with $q := |V(\mathcal{P})|/2$, where the action of $G(\mathcal{P})$ on $D_s^q$ ($=W$) is induced by the action on $\mathcal{G}(s)$. In particular, $2^{\mathcal{P},\mathcal{G}(s)}$ is finite if and only if $\mathcal{P}$ is finite, in which case $$\big|G(2^{\mathcal{P},\mathcal{G}(s)})\big| = |D_s|^q \cdot |G(\mathcal{P})| = (2s)^{|V(\mathcal{P})|/2}|G(\mathcal{P})|;$$
  \item If $s$ is even and $\mathcal{P}$ has only finitely many vertices, then $2^{\mathcal{P},\mathcal{G}(s)}$ is also centrally symmetric.
 \end{enumerate}
\end{theorem}

Notice that if $\mathcal{P}$ is the Coxeter group $[p_1,\ldots,p_{n-1}]$ factorized by a set of relations $R$, then $2^{\mathcal{P},\mathcal{G}(s)}$ is a Coxeter group $[4,p_1,\ldots,p_{n-1}]$ factorized by the relations in $R$ and the extra relations 
$$(\rho_0\alpha\rho_0\alpha)^s = id,$$
where $\alpha$ is the central involution of $\mathcal{P}$, and
$$(\rho_0\tau^{-1}\rho_0\tau)^2 = id$$
for all $\tau\in G(\mathcal{P})$ such that $\tau\neq \alpha$ and $\{F_0,F_0\tau\}$ give distinct diagonal classes of $\mathcal{P}$.

When $\mathcal{P}$ is non-degenerate, this construction gives a non-degenerate polytope, as expressed the next lemma.

\begin{lemma}\label{2KGlattice}[pp. 264 of \cite{ARP}]
Let $\mathcal{P}$ be a centrally symmetric regular $n$-polytope of type $\{p_1,\ldots,p_{n-1}\}$ with $p_1\geq 3$. If the poset of $\mathcal{P}$ is a lattice, then the poset of $2^{\mathcal{P},\mathcal{G}(s)}$ is a lattice. 
\end{lemma}

\section{Polytopes ${2}^{\mathcal{P},\mathcal{G}(s)}$ and Hypertopes $\mathcal{H}({2}^{\mathcal{P},\mathcal{G}(s)})$ when $\mathcal{P}$ is a $2p$-gon, $n$-cube or $n$-orthoplex}\label{rank2&n}

In Table~\ref{tab:centralinvo} we were introduced to the centrally symmetric polytopes $\mathcal{P}$ that we will consider in this paper. There are three infinite families of polytopes given in that table: the $2p$-gons with type $\{2p\}$, the $n$-cube with type $\{4,3^{n-2}\}$ and the $n$-orthoplex with type $\{3^{n-2},4\}$. In the following sections, we will construct extensions of these polytopes and then apply the halving operation as defined in Section~\ref{halv} to obtain families of hypertopes.
Moreover, families of proper regular toroidal hypertopes $\{\nfrac{3}{3},3^{n-3},4\}_{(2s,0^{n-1})}$ will be given for an arbitrary rank and $s$, extending the results of \cite{rank4toroidal} and \cite{teroasia2020}, where the duals of the hypertopes $\left\{\nfrac{3}{3},4\right\}_{(2s,0,0)}$ and $\left\{\nfrac{3}{3},3^{n-3},4\right\}_{(4,0^{n-1})}$ are presented, respectively.

\subsection{The polytope ${2 }^{\{2p\},\mathcal{G}(s)}$ and hypertope $\mathcal{H}({2}^{\{2p\},\mathcal{G}(s)})$}\label{ss:2p}

Consider the following polytopes, defined as below.
\begin{definition}\label{def:skewpoly}\cite[Section 7B]{ARP}
 Let $2\leq j \leq k := \lfloor\frac{1}{2}q\rfloor$. Then, we define the polytope $ \mathcal{P} := \{p,q\ |\ h_2,\ldots,h_k\}$
 such that its automorphism group $G(\mathcal{P})$ has the following presentation
 \begin{equation*}\begin{matrix}
  G(\mathcal{P}) :=& \langle\ \rho_0, \rho_1, \rho_2\ |\  \rho_0^2 = \rho_1^2 = \rho_2^2 = (\rho_0\rho_1)^p = (\rho_1\rho_2)^q = (\rho_0\rho_2)^2 = id, \\
    &  \{ (\rho_0\rho_1(\rho_2\rho_1)^{j-1})^{h_j} = id,\mbox{ for } 2\leq j \leq k\}\, \rangle
                                                                                              \end{matrix}
 \end{equation*}
\end{definition}

Let $\mathcal{P}$ be the polygons with even number of vertices, i.e. of type $\{2p\}$, for $p\geq 2$.
From Corollary 8C7 of \cite{ARP}, we have the following result.
\begin{corollary}\label{cor:2p_polytope}\cite[Corollary 8C7]{ARP}
 Let $2\leq p < \infty$ and $2\leq s < \infty$. Then ${2 }^{\{2p\},\mathcal{G}(s)} = \{4,2p\ |\ 4^{p-2}, 2s\}$, with group $D_s^p\rtimes D_{2p}$, of order $(2s)^p\cdot 4p$. If $p=2$, this is the torus map $\{4,4\}_{(2s,0)}$, with group $(D_s\times D_s)\rtimes D_4$ of order $32s^2$.
\end{corollary}

We write $4^{p-2}$ to mean a row of 4's of size $(p-2)$. Using this result and Definition~\ref{def:skewpoly}, we have the following proposition.

\begin{proposition}\label{prop:2p_pol_presentation}
 Let $2\leq p < \infty$ and $2\leq s < \infty$. Then the group $G({2 }^{\{2p\},\mathcal{G}(s)}) = [4,2p\ |\ 4^{p-2}, 2s]$ has the following presentation
  $$G({2 }^{\{2p\},\mathcal{G}(s)}) := \langle\ \rho_0, \rho_1, \rho_2\ |\  \rho_0^2 = \rho_1^2 = \rho_2^2 = (\rho_0\rho_1)^4 = (\rho_1\rho_2)^{2p} = (\rho_0\rho_2)^2 = id,$$ $$
            \{ (\rho_0\rho_1(\rho_2\rho_1)^{j-1})^{4} = id,\mbox{ for } 2\leq j \leq p-1\}, (\rho_0\rho_1(\rho_2\rho_1)^{p-1})^{2s} = id\ \rangle. $$                                                                          
\end{proposition}
\begin{proof}
 The proof follows from Corollary~\ref{cor:2p_polytope} and Definition~\ref{def:skewpoly}.
\end{proof}

From the polytopes of the previous proposition, we derive a family of polytopes using the halving operation.

\begin{proposition}
 Let $2\leq p < \infty$, $2\leq s < \infty$. The incidence system $$\mathcal{H}({2 }^{\{2p\},\mathcal{G}(s)}) = \Gamma(H({2 }^{\{2p\},\mathcal{G}(s)}), (H_i)_{i\in \{0,1,2\}}),$$ where $H({2 }^{\{2p\},\mathcal{G}(s)}):=\langle \rho_0\rho_1\rho_0, \rho_1, \rho_2\rangle = \langle \tilde{\rho_0}, \rho_1, \rho_2\rangle $, is a regular polytope of type $\{2p,2p\}$, where its automorphism group, of size $(2s)^p\cdot 2p$, is the quotient of the Coxeter group $[2p,2p]$ by the relations $((\tilde{\rho_0}\rho_2)^{j-1}\tilde{\rho_0}\rho_1(\rho_2\rho_1)^{j-1})^{2} = id$, for $2\leq j \leq p-1$, and $((\tilde{\rho_0}\rho_2)^{p-1}\tilde{\rho_0}\rho_1(\rho_2\rho_1)^{p-1})^{s} = id$.            
\end{proposition}
\begin{proof}
 Let $2\leq p < \infty$, $2\leq s < \infty$. The fact that the incidence system of the halving group $H({2 }^{\{2p\},\mathcal{G}(s)})$ is a regular hypertope follows from the fact that ${2 }^{\{2p\},\mathcal{G}(s)}$ is non-degenerate (a lattice, by Lemma~\ref{2KGlattice}) and from Corollary~\ref{H(P)hypertopes}. Moreover, in Section 7B \cite{ARP}, it is given that the halving operation on regular polytope of type $\{4,k\}$ results in a regular polytope of type $\{k,k\}$. Let us write the relations that are not of the infinite Coxeter group $[2p,2p]$.
 
Firstly,
  \begin{equation*}
  \begin{split}
   id & = (\rho_0\rho_1(\rho_2\rho_1)^{j-1})^{4}\\ 
      & = (\rho_0\rho_1(\rho_2\rho_1)^{j-1}\rho_0\rho_1(\rho_2\rho_1)^{j-1})^2 \\
      & = (\rho_0\rho_1(\rho_0\rho_2\rho_0\rho_1)^{j-1}\rho_0\rho_1(\rho_2\rho_1)^{j-1})^2 \\
      & = ((\rho_0\rho_1\rho_0\rho_2)^{j-1}\rho_0\rho_1\rho_0\rho_1(\rho_2\rho_1)^{j-1})^2\\
      & = ((\tilde{\rho_0}\rho_2)^{j-1}\tilde{\rho_0}\rho_1(\rho_2\rho_1)^{j-1})^2.
  \end{split} 
  \end{equation*}
For the relation $(\rho_0\rho_1(\rho_2\rho_1)^{p-1})^{2s} = id$, similar arguments give 
$$ id = (\rho_0\rho_1(\rho_2\rho_1)^{p-1})^{2s} = ((\tilde{\rho_0}\rho_2)^{p-1}\tilde{\rho_0}\rho_1(\rho_2\rho_1)^{p-1})^s.$$  
\end{proof}

Notice that, if $p=2$, the regular hypertope obtained from the halving of $\{4,4\}_{(2s,0)}$ is the regular map $\{4,4\}_{(s,s)}$.

\subsection{The polytope ${2 }^{\{3^{n-2},4\},\mathcal{G}(s)}$ and hypertope $\mathcal{H}({2}^{\{3^{n-2},4\},\mathcal{G}(s)})$}\label{northoplex-ext}

Let $\mathcal{P}$ be the $n$-orthoplex, with $n\geq 3$.
From Corollary 8C6 of \cite{ARP}, we have the following result.

\begin{corollary}\label{cor:ncubic_tess}\cite[Corollary 8C6]{ARP}
 Let $n\geq 3$ and $2\leq s < \infty$. The polytope ${2 }^{\{3^{n-2},4\},\mathcal{G}(s)}$ is the cubical regular $(n+1)$-toroid $\{4,3^{n-2},4\}_{(2s,0^{n-1})}$, with group $D_s^n\rtimes [3^{n-2},4]$ of order $(4s)^n n!$.
\end{corollary}

The defining relations of the regular polytope $\{4,3^{n-2},4\}_{(2s,0^{n-1})}$ are those given by its Schl\"{a}fli type and the extra relation \cite[\mbox{Section 6D}]{ARP} $$(\rho_0\rho_1\rho_2\ldots\rho_{n-1}\rho_n\rho_{n-1}\ldots\rho_2\rho_1)^{2s} = id.$$

Now, using the halving operation, we get a family of proper regular toroidal hypertopes, as shown in the following result.

\begin{proposition}\label{tildeB3}
 Let $n\geq 3$, $2\leq s < \infty$. The incidence system $$\mathcal{H}({2 }^{\{3^{n-2},4\},\mathcal{G}(s)}) = \Gamma(H({2 }^{\{3^{n-2},4\},\mathcal{G}(s)}), (H_i)_{i\in \{0,\ldots,n\}}),$$ where $H({2 }^{\{3^{n-2},4\},\mathcal{G}(s)}):=\langle \rho_0\rho_1\rho_0, \rho_1, \ldots, \rho_n\rangle = \langle \tilde{\rho_0}, \rho_1, \ldots, \rho_n\rangle $, is a regular hypertope whose automorphism group, of size $(4s)^{n-1}(2s)n!$, is the quotient of the Coxeter group with diagram 
 $$
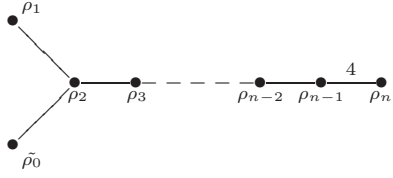
\xymatrix@-1.5pc{*{\bullet}\ar@{-}[rrdd]^(0.01){\rho_1} \\
\\
 &&*{\bullet}\ar@{-}[rr]_(0.05){\rho_2}&&*{\bullet}\ar@{--}[rrrr]_(0.01){\rho_3}&&&&*{\bullet}\ar@{-}[rr]_(0.01){\rho_{n-2}}&&*{\bullet}\ar@{-}[rr]^4_(0.01){\rho_{n-1}}_(0.99){\rho_n}&&*{\bullet}\\
 \\
 *{\bullet}\ar@{-}[rruu]_(0.01){\tilde{\rho_0}}\\
 } $$
 factorized by $$(\tilde{\rho_0}\rho_2\rho_3\ldots\rho_{n-1}\rho_n\rho_{n-1}\ldots\rho_3\rho_2\rho_1)^{2s} = id.$$
\end{proposition}
\begin{proof}
  The incidence system of the halving group $H({2 }^{\{3^{n-2},4\},\mathcal{G}(s)})$ is a regular hypertope since the poset of the polytope $\{3^{n-2},4\}$ is a lattice, making ${2 }^{\{3^{n-2},4\},\mathcal{G}(s)}$ non-degenerate (by Lemma~\ref{2KGlattice}), which is under the conditions of Corollary~\ref{H(P)hypertopes}. The relations of the Coxeter diagram above follow naturally from the definition of the halving operation.
  
  Consider now the extra relation $$(\rho_0\rho_1\rho_2\ldots\rho_{n-1}\rho_n\rho_{n-1}\ldots\rho_2\rho_1)^{2s} = id$$ of $\{4,3^{n-2},4\}_{(2s,0^{n-1})}$.
  Then,
   \begin{equation*}
  \begin{split}
   id & =(\rho_0\rho_1\rho_2\ldots\rho_{n-1}\rho_n\rho_{n-1}\ldots\rho_2\rho_1)^{2s}\\
      & = (\rho_0\rho_1\rho_2\ldots\rho_{n-1}\rho_n\rho_{n-1}\ldots\rho_2\rho_1\rho_0\rho_1\rho_2\ldots\rho_{n-1}\rho_n\rho_{n-1}\ldots\rho_2\rho_1)^s\\
      & = (\rho_0\rho_1\rho_2\ldots\rho_{n-1}\rho_n\rho_{n-1}\ldots\rho_2\rho_0\rho_1\rho_0\rho_1\rho_0\rho_2\ldots\rho_{n-1}\rho_n\rho_{n-1}\ldots\rho_2\rho_1)^s\\
      & =  (\rho_0\rho_1\rho_0\rho_2\ldots\rho_{n-1}\rho_n\rho_{n-1}\ldots\rho_2\rho_1\rho_0\rho_1\rho_0\rho_2\ldots\rho_{n-1}\rho_n\rho_{n-1}\ldots\rho_2\rho_1)^s \\
      & =  (\tilde{\rho_0}\rho_2\ldots\rho_{n-1}\rho_n\rho_{n-1}\ldots\rho_2\rho_1\tilde{\rho_0}\rho_2\ldots\rho_{n-1}\rho_n\rho_{n-1}\ldots\rho_2\rho_1)^s\\
      & = (\tilde{\rho_0}\rho_2\ldots\rho_{n-1}\rho_n\rho_{n-1}\ldots\rho_2\rho_1)^{2s}
  \end{split} 
  \end{equation*}
\end{proof}
 Following the notation of Ens \cite{rank4toroidal} and Weiss and Montero \cite{teroasia2020}, we denote these regular toroidal hypertopes by $\left\{\nfrac{3}{3},3^{n-3},4\right\}_{(2s,0^{n-1})}$.
 
\subsection{The polytope ${2 }^{\{4,3^{n-2}\},\mathcal{G}(s)}$ and hypertope $\mathcal{H}({2}^{\{4,3^{n-2}\},\mathcal{G}(s)})$}\label{ncubehyper}

Consider the $n$-cube with automorphism group $[4,3^{n-2}] = \langle \tau_0,\ldots,\tau_{n-1}\rangle$, with $n\geq 3$.
 
 Let $(v_1,v_2,\ldots,v_n)$ be the coordinates of a vertex of the $n$-cube in an Euclidean space and let $v_i\in\{\pm 1\}$. In addition, let 
 \begin{equation*}
  \begin{split}
   (v_1,v_2,\ldots,v_n)\tau_0 &:= (-v_1,v_2,\ldots,v_n) \\
   (v_1,\ldots, v_{j-1}, v_j, v_{j+1}, v_{j+2},\ldots,v_n)\tau_j &:= (v_1,\ldots, v_{j-1}, v_{j+1},v_j, v_{j+2},\ldots,v_n),\\
  \end{split}
 \end{equation*}
  for $j\in\{1,\ldots,n-1\}$.
 Let $F_0 := (1^n)$ be the vertex having all coordinates equal to 1, and let $\beta := \tau_0\tau_1\tau_2\ldots\tau_{n-2}\tau_{n-1}$. Then,
 $$ F_0 \beta = (1^{n-1},-1)$$
 where  $1^{n-1}$ means we have a row of $+1$'s of size $(n-1)$. Moreover, it is easily seen that
 $$ F_0 \beta^i = (1^{n-i},-1^i)$$
 Particularly, $$F_0\beta^n = (-1^n).$$
 Thus $\beta^n$ is clearly the central involution of the $n$-cube.
 
 \begin{lemma}\label{ncubediagonalclasses}
The $n$-cube has exactly $n$ diagonal classes which can be represented by $\{F_0, F_0\beta^i\}$, for $1\leq i\leq n$, where $F_0$ is a vertex, $\beta = \tau_0\tau_1\tau_2\ldots\tau_{n-2}\tau_{n-1}$, and $F_0\beta^i$ is the vertex of the action of $\beta^i$ on the vertex $F_0$. 
 \end{lemma}
\begin{proof}
Consider the construction of the vertices of the cube as above and let $F_0 := (1^n)$.
Let $1\leq i,j\leq n$ and consider the vertices $F_0 \beta^i = (1^{n-i},-1^i)$ and  $F_0 \beta^j = (1^{n-j},-1^j)$. Suppose that the diagonals $\{F_0, F_0\beta^i\}$ and $\{ F_0, F_0\beta^j\}$ are in the same diagonal class. Then, they share the same square length as their diagonal class representative
$$ || F_0 - F_0\beta^i||^2 = || F_0 - F_0\beta^j||^2.$$
Hence,
\begin{equation*}
 \begin{split}
  || F_0 - F_0\beta^i||^2 &= || F_0 - F_0\beta^j|| \Leftrightarrow\\
 \Leftrightarrow  || (0^{n-i}, 2^i )|| ^2 &= || (0^{n-j}, 2^j )||^2 \Rightarrow\\
 \Leftrightarrow i &= j .
 \end{split}
\end{equation*}
Since there are $n$ distinct diagonal classes of the $n$-cube \cite[Section 5B]{ARP} and we can represent $n$ distinct diagonal classes as above, we have proven the statement of the lemma.
\end{proof}

 With the above lemma, we are able to give the relations of the group of automorphisms of ${2 }^{\{4,3^{n-2}\},\mathcal{G}(s)}$.
 
 \begin{corollary}\label{cor:ncube_polytope}
  Let $n\geq 3$ and $2\leq s < \infty$. Then ${2}^{\{4,3^{n-2}\},\mathcal{G}(s)}$ is a $(n+1)$-polytope with type $\{4,4,3^{n-2}\}$ and automorphism group $D_s^{2^{n-1}}\rtimes [4, 3^{n-2}]$ of order $(2s)^{2^{n-1}} 2^n n!$ with the relations given by its Coxeter diagram and the following extra relations $$(\rho_0\beta^{-i}\rho_0\beta^i)^2 = id\mbox{ for } 2 \leq i \leq n-1$$$$(\rho_0\beta^n\rho_0\beta^n)^s = id,$$ where $\beta = \rho_1\rho_2\ldots\rho_n$. Moreover, its toroidal residue is the map $\{4,4\}_{(4,0)}$.
 \end{corollary}
 \begin{proof}
  The polytopes above are obtained by Theorem~\ref{2KG(s)} and, when $i=n$, $\beta^n$ is the central involution of the polytope $\{4,3^{n-2}\}$, meaning that 
  $$(\rho_0\beta^n\rho_0\beta^n)^s = id.$$
  The remaining extra relations of the statement of this corollary come from diagonal classes of improper branches of the diagram $\mathcal{G}(s)$.
  Particularly, when $i=1$, we have
  $$id = (\rho_0 \beta^{-1}\rho_0\beta)^2  = (\rho_0 \rho_{n-1}\ldots\rho_2\rho_1 \rho_0 \rho_1 \rho_2 \ldots \rho_n)^2 $$
  which implies that 
  $$ id =  (\rho_0 \rho_1 \rho_0 \rho_1)^2 = (\rho_0\rho_1)^4,$$
a relation given by the type of the polytope. Hence, $(\rho_0\beta^n\rho_0\beta^n)^s = id$ and $(\rho_0\beta^{-i}\rho_0\beta^i)^2 = id$, for $2 \leq i \leq n-1$, are the only extra relations needed to define the automorphism group of ${2}^{\{4,3^{n-2}\},\mathcal{G}(s)}$.
  
  To prove that the toroidal residue is the map $\{4,4\}_{(4,0)}$, observe that from 
  $id = (\rho_0\beta^{-2}\rho_0\beta^2)^2$ we have 
  $$ id = (\rho_0 \rho_1 \rho_{n-1}\ldots\rho_2\rho_1 \rho_0 \rho_1 \rho_2 \ldots \rho_n \rho_1 )^2$$
  which implies
  $$id = (\rho_0 \rho_1 \rho_2\rho_1 \rho_0 \rho_1 \rho_2 \rho_1 )^2 = (\rho_0\rho_1\rho_2\rho_1)^4,$$
  showing that the toroidal residue is the map $\{4,4\}_{(4,0)}$.
 \end{proof}

 If $n=3$, the resulting abstract polytopes are quotients of the locally toroidal polytope of type $\{4,4,3\}$, satisfying the following relations
 $$(\rho_0\rho_1\rho_2\rho_1)^4 = (\rho_0(\rho_1\rho_2\rho_3)^3)^{2s} = id,$$
 which do not give an universal locally toroidal polytope. Therefore, these polytopes do not appear in \cite{ARP}.

 Let us construct the regular hypertopes corresponding to this family. As before, we will use the halving operation.
 
 \begin{proposition}\label{pp:hyper_ncube}
  Let $n\geq 3$ and $2\leq s < \infty$. The incidence system $$\mathcal{H}({2 }^{\{4,3^{n-2}\},\mathcal{G}(s)}) = \Gamma(H({2 }^{\{4,3^{n-2}\},\mathcal{G}(s)}), (H_i)_{i\in \{0,\ldots,n\}}),$$ where $H({2 }^{\{4,3^{n-2}\},\mathcal{G}(s)}):=\langle \rho_0\rho_1\rho_0, \rho_1, \ldots, \rho_n\rangle = \langle \tilde{\rho_0}, \rho_1, \ldots, \rho_n\rangle $, is a regular hypertope and its automorphism group, of size $(2s)^{2^{n-1}} 2^{n-1} n!$, has the relations given by its Coxeter diagram
$$
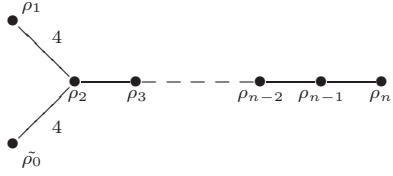
\xymatrix@-1.5pc{*{\bullet}\ar@{-}[rrdd]^(0.01){\rho_1}^4 \\
\\
 &&*{\bullet}\ar@{-}[rr]_(0.05){\rho_2}&&*{\bullet}\ar@{--}[rrrr]_(0.01){\rho_3}&&&&*{\bullet}\ar@{-}[rr]_(0.01){\rho_{n-2}}&&*{\bullet}\ar@{-}[rr]_(0.01){\rho_{n-1}}_(0.99){\rho_n}&&*{\bullet}\\
 \\
 *{\bullet}\ar@{-}[rruu]_(0.01){\tilde{\rho_0}}_ 4\\
 } $$
 and the extra relations $(\tilde{\beta}^{-i}\beta^i)^2 = id$, for $2 \leq i \leq n-1$, and $(\tilde{\beta}^n\beta^n)^s = id$, where $\beta = \rho_1\rho_2\ldots\rho_n$ and $\tilde{\beta}= \tilde{\rho_0}\rho_2\ldots\rho_n$. 
Moreover the toroidal residue is the map $\{4,4\}_{(2,2)}$.
 \end{proposition}
\begin{proof}Let $n\geq 3$ and $2\leq s < \infty$. The incidence system of the halving group $H({2 }^{\{4,3^{n-2}\},\mathcal{G}(s)})$ is a regular hypertope by Corollary~\ref{H(P)hypertopes} since the poset of the polytope $\{4,3^{n-2}\}$ is a lattice, making ${2 }^{\{4,3^{n-2}\},\mathcal{G}(s)}$ non-degenerate, by Lemma~\ref{2KGlattice}.
 
 Then, if we denote $\beta = \rho_1\rho_2\ldots\rho_n$ and $\tilde{\beta}= \tilde{\rho_0}\rho_2\ldots\rho_n$, we have
 \begin{equation*}
   \begin{split}
    id & = (\rho_0\beta^{-i}\rho_0\beta^i)^k \\
        & = (\rho_0 (\rho_{n-1}\ldots\rho_2\rho_1)^i \rho_0 (\rho_1 \rho_2 \ldots \rho_n)^i)^k \\
        & = ((\rho_{n-1}\ldots\rho_2\rho_0\rho_1\rho_0)^i  (\rho_1 \rho_2 \ldots \rho_n)^i )^k\\
        & = ( (\rho_{n-1}\ldots\rho_2\tilde{\rho_0})^i (\rho_1 \rho_2 \ldots \rho_n^i )^k\\
        & = (\tilde{\beta}^{-i}\beta^i)^k,
   \end{split}
  \end{equation*}
 where $k=2$ if $2\leq i \leq n-1$, and $k=s$ if $i=n$. Moreover, we have that $\tilde{\beta}^n = \tilde{\beta}^{-n}$, meaning that $$(\tilde{\beta}^{-n}\beta^n)^s = (\tilde{\beta}^{n}\beta^n)^s.$$
  
 Let us prove that the toroidal residue is the map $\{4,4\}_{(2,2)}$.
 Consider the translations $u := \tilde{\rho_0}\rho_2\rho_1\rho_2$ and $g := (\tilde{\rho_0}\rho_2\rho_1)^2$ of the toroidal map residue $\{4,4\}$ of the above hypertope. Then, we have that 
 $$ u =  \tilde{\rho_0}\rho_2\rho_1\rho_2 = \rho_0\rho_1\rho_0\rho_2\rho_1\rho_2 = \rho_0\rho_1\rho_2\rho_0\rho_1\rho_2 = (\rho_0\rho_1\rho_2)^2,$$
 which is a translation of order 4 of the toroidal residue $\{4,4\}_{(4,0)}$ of ${2 }^{\{4,3^{n-2}\},\mathcal{G}(s)}$.
 Furthermore, we have that 
 $$g = (\tilde{\rho_0}\rho_2\rho_1)^2 = (\rho_0\rho_1\rho_0\rho_2\rho_1)^2 = \rho_0\rho_1\rho_2\rho_0\rho_1\rho_0\rho_1\rho_0\rho_2\rho_1 = \rho_0\rho_1\rho_2\rho_1\rho_0\rho_1\rho_2\rho_1,$$
 which is a conjugate of $\rho_0\beta^{-2}\rho_0\beta^2$, meaning $o(g) = 2$. Since $o(u)=4$ and $o(g)=2$, then the toroidal residue of our regular hypertope is the map $\{4,4\}_{(2,2)}$.
 
\end{proof}

  Particularly, when $n=3$,  the regular hypertope of type $\left\{\nfrac{4}{4},3 \right\}$ given by Proposition~\ref{pp:hyper_ncube} is locally toroidal, with toroidal residue $\{4,4\}_{(2,2)}$, and satisfies the relation $((\tilde{\rho_0}\rho_2\rho_3)^3(\rho_1\rho_2\rho_3)^3)^s=id$.

\section{Polytopes ${2}^{\mathcal{P},\mathcal{G}(s)}$ and Hypertopes $\mathcal{H}({2}^{\mathcal{P},\mathcal{G}(s)})$ when $\mathcal{P}$ has rank $3$ or $4$}\label{rank3&4}

In this section we consider that $\mathcal{P}$ is one of the remaining regular polytopes of Table~\ref{tab:centralinvo}: the icosahedron, the dodecahedron, the 24-cell, the 600-cell and the 120-cell. In what follow, similarly to the previous section, we construct extensions of these polytopes and then we apply the halving operation to obtain regular hypertopes. 
In \cite{teroasia2020}, two locally spherical regular hypertopes of hyperbolic type are given: $\left\{\nfrac{3}{3},5\right\}$, with automorphism group of order $60 \cdot 2^{12}$,  and $\left\{\nfrac{3}{3},3,5\right\}$, with automophism group of order $7200 \times 2^{120}$. These two hypertopes will correspond to our hypertopes $\mathcal{H}({2}^{\{3,5\},\mathcal{G}(2)})$ and $\mathcal{H}({2}^{\{3,3,5\},\mathcal{G}(2)})$, respectively. Here, we will give an infinite family of these hypertopes. In addition, we will give a family of hypertopes of type $\left\{\nfrac{3}{3},4,3\right\}$ with toroidal residue $\left\{\nfrac{3}{3},4\right\}_{(4,0,0)}$. 
Most proofs of the following results will be omitted as they follow the same ideas present in the proofs of Corollary~\ref{cor:ncube_polytope} and Proposition~\ref{pp:hyper_ncube}.

\subsection{The polytope ${2 }^{\{3,5\},\mathcal{G}(s)}$ and hypertope $\mathcal{H}({2}^{\{3,5\},\mathcal{G}(s)})$}\label{ss:35}

Let $\mathcal{P}$ be the icosahedron with automorphism group $G := \langle \tau_0, \tau_1, \tau_2\rangle$. The icosahedron has three distinct diagonal classes, which can be determined computationally with GAP\cite{GAP}: $\{F_0, F_0\beta\}$, $\{F_0, F_0\beta^3\}$ and $\{F_0, F_0\beta^5\}$, where $\beta := \tau_0\tau_1\tau_2$. The vertices of the latter diagonal class are antipodal. In fact, the diagonals $\{F_0, F_0\beta\}$ and $\{F_0, F_0\beta^2\}$ are in the same diagonal class, since the double $G_0$-cosets coincide, that is,
\begin{equation*}
\begin{split}
 G_0 \beta^2G_0 &= G_0 (\tau_0\tau_1\tau_2)^2G_0 = G_0 \tau_0\tau_1\tau_2\tau_0G_0 =\\ 
 &= G_0 \tau_0\tau_1\tau_0\tau_2G_0 = G_0 \tau_1\tau_0\tau_1\tau_2G_0 = G_0 \beta G_0.
 \end{split}
\end{equation*}
The same can be proven for the diagonals $\{F_0, F_0\beta^3\}$ and $\{F_0, F_0\beta^4\}$.
With this, we can provide the polytope ${2 }^{\{3,5\},\mathcal{G}(s)}$ and the hypertope $\mathcal{H}({2 }^{\{3,5\},\mathcal{G}(s)})$.

\begin{corollary}\label{cor:435}
 Let $2\leq s < \infty$. Then ${2}^{\{3,5\},\mathcal{G}(s)}$ is a $4$-polytope of type $\{4,3,5\}$ with automophism group $D_s^{6}\rtimes [3,5]$ of order $ 120 \cdot (2s)^{6} $. Moreover, the group $G({2}^{\{3,5\},\mathcal{G}(s)}) := \langle \rho_0, \rho_1, \rho_2, \rho_3\rangle $ is the quotient of the Coxeter group $[4,3,5]$ by the relations $(\rho_0\beta^{-3}\rho_0\beta^3)^2 = id$ and $(\rho_0\beta^5\rho_0\beta^5)^s = id$, where $\beta = \rho_1\rho_2\rho_3$.
\end{corollary}

\begin{proposition}
 Let $2\leq s < \infty$. The incidence system $$\mathcal{H}({2 }^{\{3,5\},\mathcal{G}(s)}) = \Gamma(H({2 }^{\{3,5\},\mathcal{G}(s)}), (H_i)_{i\in \{0,\ldots,3\}}),$$ where $H({2 }^{\{3,5\},\mathcal{G}(s)}):=\langle \rho_0\rho_1\rho_0, \rho_1, \rho_2, \rho_3\rangle = \langle \tilde{\rho_0}, \rho_1, \rho_2, \rho_3\rangle$, is a regular hypertope and its automorphism group, of size $ 60 \cdot (2s)^{6} $, is the quotient of the Coxeter group with diagram
$$
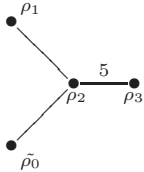
\xymatrix@-1.5pc{*{\bullet}\ar@{-}[rrdd]^(0.01){\rho_1} \\
\\
 &&*{\bullet}\ar@{-}[rr]_(0.05){\rho_2}_(0.99){\rho_3}^5&&*{\bullet}\\
 \\
 *{\bullet}\ar@{-}[rruu]_(0.01){\tilde{\rho_0}}\\
 } $$
 factorized by $(\tilde{\beta}^{-3}\beta^3)^2 = id$ and $(\tilde{\beta}^5\beta^5)^s = id$, where $\beta := \rho_1\rho_2\rho_3$ and $\tilde{\beta} := \tilde{\rho_0}\rho_2\rho_3$.
 
\end{proposition}

\subsection{The polytope ${2 }^{\{5,3\},\mathcal{G}(s)}$ and hypertope $\mathcal{H}({2}^{\{5,3\},\mathcal{G}(s)})$}

Let $\mathcal{P}$ by the dual of the icosahedron, the dodecahedron. Using GAP\cite{GAP} and the double coset action described in equation~\ref{eq:doublecosetdiagonal}, we can determine the diagonal classes of the dodecahedron: $\{F_0, F_0\beta^i\}$, for $1\leq i\leq 5$, where $\beta = \tau_0\tau_1\tau_2$ is an element of the group $[5,3] = \langle \tau_0, \tau_1, \tau_2\rangle$. 
As before, we give the polytope ${2 }^{\{5,3\},\mathcal{G}(s)}$ and the hypertope $\mathcal{H}({2 }^{\{5,3\},\mathcal{G}(s)})$.

\begin{corollary}\label{cor:453}
  Let $2\leq s < \infty$. Then ${2}^{\{5,3\},\mathcal{G}(s)}$ is a family of $4$-polytopes with type $\{4,5,3\}$ and automorphism group $D_s^{10}\rtimes [5,3]$ of order $ 120 \cdot (2s)^{10} $. Moreover, the group $G({2}^{\{5,3\},\mathcal{G}(s)}) := \langle \rho_0, \rho_1, \rho_2, \rho_3\rangle $ is the quotient of the Coxeter group $[4,5,3]$ by the relations $(\rho_0\beta^{-i}\rho_0\beta^i)^2 = id$, for $2\leq i\leq 4$, and $(\rho_0\beta^5\rho_0\beta^5)^s = id$, where $\beta = \rho_1\rho_2\rho_3$.
\end{corollary}

\begin{proposition}
 Let $2\leq s < \infty$. The incidence system $$\mathcal{H}({2 }^{\{5,3\},\mathcal{G}(s)}) = \Gamma(H({2 }^{\{5,3\},\mathcal{G}(s)}), (H_i)_{i\in \{0,\ldots,3\}}),$$ where $H({2 }^{\{5,3\},\mathcal{G}(s)}):=\langle \rho_0\rho_1\rho_0, \rho_1, \rho_2, \rho_3\rangle = \langle \tilde{\rho_0}, \rho_1, \rho_2, \rho_3\rangle$, is a regular hypertope and its automorphism group, of size $60 \cdot (2s)^{10}$, has the relations given by its Coxeter diagram
$$
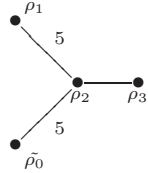
\xymatrix@-1.5pc{*{\bullet}\ar@{-}[rrdd]^(0.01){\rho_1}^5 \\
\\
 &&*{\bullet}\ar@{-}[rr]_(0.05){\rho_2}_(0.99){\rho_3}&&*{\bullet}\\
 \\
 *{\bullet}\ar@{-}[rruu]_(0.01){\tilde{\rho_0}}_5\\
 } $$
 and the extra relations $(\tilde{\beta}^{-i}\beta^i)^2 = id$, for $2 \leq i \leq 4$, and $(\tilde{\beta}^5\beta^5)^s = id$, where $\beta := \rho_1\rho_2\rho_3$ and $\tilde{\beta} := \tilde{\rho_0}\rho_2\rho_3$.
 
\end{proposition}

\subsection{The polytope ${2 }^{\{3,4,3\},\mathcal{G}(s)}$ and hypertope $\mathcal{H}({2}^{\{3,4,3\},\mathcal{G}(s)})$}\label{ss:343}

Let $\mathcal{P}$ be the self-dual polytope of type $\{3,4,3\}$. Using GAP\cite{GAP} we have that the 24-cell has 4 distinct diagonal classes: $\{F_0, F_0\beta^i\}$, for $i\in \{1,3,4,6\}$, where $\beta = \tau_0\tau_1\tau_2\tau_3$ is an element of the group $[3,4,3] = \langle \tau_0, \tau_1, \tau_2, \tau_3\rangle$. 
We will determine the polytope ${2 }^{\{3,4,3\},\mathcal{G}(s)}$ and regular hypertope $\mathcal{H}({2 }^{\{3,4,3\},\mathcal{G}(s)})$.

\begin{corollary}\label{cor:4343}
  Let $2\leq s < \infty$. Then ${2}^{\{3,4,3\},\mathcal{G}(s)}$ is a $5$-polytope of  type $\{4,3,4,3\}$ and automorphism group $D_s^{12}\rtimes [3,4,3]$ of order $1152\cdot (2s)^{12}$. Moreover, the group $G({2}^{\{3,4,3\},\mathcal{G}(s)}) := \langle \rho_0, \rho_1, \rho_2, \rho_3, \rho_4\rangle $ is the quotient of the locally toroidal Coxeter group $[4,3,4,3]$ factorized by the relations $(\rho_0\beta^{-i}\rho_0\beta^i)^2 = id\mbox{, for }i \in\{3,4\}$, and $(\rho_0\beta^6\rho_0\beta^6)^s = id,$ where $\beta = \rho_1\rho_2\rho_3\rho_4$ and its toroidal residue is the cubic toroid $\{4,3,4\}_{(4,0,0)}$.
\end{corollary}
\begin{proof}
 The proof follows the same idea as in Corollary~\ref{cor:ncube_polytope}. To prove that the toroidal residue is the cubic toroid $\{4,3,4\}_{(4,0,0)}$, observe that the relation $(\rho_0\beta^{-3}\rho_0\beta^3)^2 = id$ implies that $(\rho_0\rho_1\rho_2\rho_3\rho_4\rho_3\rho_2\rho_1)^4 = id$. 
\end{proof}

\begin{proposition}
 Let $2\leq s < \infty$. The incidence system $$\mathcal{H}({2 }^{\{3,4,3\},\mathcal{G}(s)}) = \Gamma(H({2 }^{\{3,4,3\},\mathcal{G}(s)}), (H_i)_{i\in \{0,\ldots,4\}}),$$ where $H({2 }^{\{3,4,3\},\mathcal{G}(s)}):=\langle \rho_0\rho_1\rho_0, \rho_1, \rho_2, \rho_3, \rho_4\rangle = \langle \tilde{\rho_0}, \rho_1, \rho_2, \rho_3, \rho_4 \rangle$, is a regular hypertope and its automorphism group, of size $576 \cdot (2s)^{12}$, has the relations given by its Coxeter diagram
$$
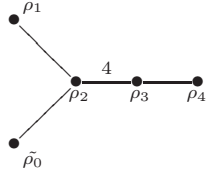
\xymatrix@-1.5pc{*{\bullet}\ar@{-}[rrdd]^(0.01){\rho_1} \\
\\
 &&*{\bullet}\ar@{-}[rr]^4_(0.05){\rho_2}&&*{\bullet}\ar@{-}[rr]_(0.05){\rho_3}_(0.99){\rho_4}&&*{\bullet}\\
 \\
 *{\bullet}\ar@{-}[rruu]_(0.01){\tilde{\rho_0}}\\
 } $$
 and the extra relations $(\tilde{\beta}^{-i}\beta^i)^2 = id\mbox{, for } i\in\{3,4\},$ and $(\tilde{\beta}^6\beta^6)^s = id,$ where $\beta := \rho_1\rho_2\rho_3\rho_4$ and $\tilde{\beta} := \tilde{\rho_0}\rho_2\rho_3\rho_4$. Moreover, its toroidal residue is $\left\{\nfrac{3}{3},4\right\}_{(4,0,0)}$.
\end{proposition}
\begin{proof}
 The proof follows the same idea as in Proposition~\ref{pp:hyper_ncube}. Moreover, as seen in the proof of Proposition~\ref{tildeB3}, we can rewrite the relation $(\rho_0\rho_1\rho_2\rho_3\rho_4\rho_3\rho_2\rho_1)^4 = id$, given in Corollary~\ref{cor:4343}, as $(\tilde{\rho_0}\rho_2\rho_3\rho_4\rho_3\rho_2\rho_1)^4 = id$, which is the factorizing relation of the hypertope $\left\{\nfrac{3}{3},4\right\}_{(4,0,0)}$.
\end{proof}

\subsection{The polytope ${2 }^{\{3,3,5\},\mathcal{G}(s)}$ and hypertope $\mathcal{H}({2}^{\{3,3,5\},\mathcal{G}(s)})$}\label{ss:335}

Let $\mathcal{P}$ be the 600-cell. The 600-cell has 8 distinct diagonal classes, which can be obtained with GAP\cite{GAP} and can be represented by $\{F_0, F_0\beta^i\}$, for $i\in \{1,4,6,7,9,10,12,15\}$, where $\beta = \tau_0\tau_1\tau_2\tau_3$ is an element of the group $[3,3,5] = \langle \tau_0, \tau_1, \tau_2, \tau_3\rangle$. 
Let us determine the polytope ${2}^{\{3,3,5\},\mathcal{G}(s)}$ and the hypertope $\mathcal{H}({2 }^{\{3,3,5\},\mathcal{G}(s)})$.

\begin{corollary}\label{cor:4335}
  Let $2\leq s < \infty$. Then ${2}^{\{3,3,5\},\mathcal{G}(s)}$ is a $5$-polytope of type $\{4,3,3,5\}$ and automophism group $D_s^{60}\rtimes [3,3,5]$ of order $ 14400 \cdot (2s)^{60}$. Moreover, the group $G({2}^{\{3,3,5\},\mathcal{G}(s)}) := \langle \rho_0, \rho_1, \rho_2, \rho_3, \rho_4\rangle $ is the quotient of the Coxeter group $[4,3,3,5]$ by the relations $(\rho_0\beta^{-i}\rho_0\beta^i)^2 = id$, for $i \in\{4,6,7,9,10,12\}$, and $(\rho_0\beta^{15}\rho_0\beta^{15})^s = id$, where $\beta = \rho_1\rho_2\rho_3\rho_4$.
\end{corollary}

\begin{proposition}
 Let $2\leq s < \infty$. The incidence system $$\mathcal{H}({2 }^{\{3,3,5\},\mathcal{G}(s)}) = \Gamma(H({2 }^{\{3,3,5\},\mathcal{G}(s)}), (H_i)_{i\in \{0,\ldots,4\}}),$$ where $H({2 }^{\{3,3,5\},\mathcal{G}(s)}):=\langle \rho_0\rho_1\rho_0, \rho_1, \rho_2, \rho_3, \rho_4\rangle = \langle \tilde{\rho_0}, \rho_1, \rho_2, \rho_3, \rho_4 \rangle$, is a regular hypertope and its automorphism group, of size $7200 \cdot (2s)^{60}$, has the relations given by its Coxeter diagram
$$
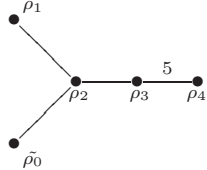
\xymatrix@-1.5pc{*{\bullet}\ar@{-}[rrdd]^(0.01){\rho_1} \\
\\
 &&*{\bullet}\ar@{-}[rr]_(0.05){\rho_2}&&*{\bullet}\ar@{-}[rr]^5_(0.05){\rho_3}_(0.99){\rho_4}&&*{\bullet}\\
 \\
 *{\bullet}\ar@{-}[rruu]_(0.01){\tilde{\rho_0}}\\
 } $$
 and the extra relations $(\tilde{\beta}^{-i}\beta^i)^2 = id$, for $i\in\{4,6,7,9,10,12\}$, and $(\tilde{\beta}^{15}\beta^{15})^s = id$, where $\beta := \rho_1\rho_2\rho_3\rho_4$ and $\tilde{\beta} := \tilde{\rho_0}\rho_2\rho_3\rho_4$.
\end{proposition}

\subsection{The polytope ${2 }^{\{5,3,3\},\mathcal{G}(s)}$ and hypertope $\mathcal{H}({2}^{\{5,3,3\},\mathcal{G}(s)})$}

Lastly, let $\mathcal{P}$ be the 120-cell polytope. As previously, we can determine with the help of GAP\cite{GAP} that the 120-cell has 15 distinct diagonal classes, using the double coset action described in equation~\ref{eq:doublecosetdiagonal}. These diagonal classes can be represented by $\{F_0, F_0\beta^i\}$, for $1\leq i\leq 15$, where $\beta = \tau_0\tau_1\tau_2\tau_3$ is an element of the group $[5,3,3] = \langle \tau_0, \tau_1, \tau_2, \tau_3\rangle$. 
Let us determine the polytope ${2}^{\{5,3,3\},\mathcal{G}(s)}$ and hypertope $\mathcal{H}({2 }^{\{5,3,3\},\mathcal{G}(s)})$.

\begin{corollary}\label{cor:4533}
  Let $2\leq s < \infty$. Then ${2}^{\{5,3,3\},\mathcal{G}(s)}$ is a $5$-polytope of type $\{4,5,3,3\}$ and automophism group $D_s^{300}\rtimes [5,3,3]$ of order $ 14400\cdot(2s)^{300}$. Moreover, the group $G({2}^{\{3,3,5\},\mathcal{G}(s)}) := \langle \rho_0, \rho_1, \rho_2, \rho_3, \rho_4\rangle $ is the quotient of the Coxeter group $[4,5,3,3]$ by the relations $(\rho_0\beta^{-i}\rho_0\beta^i)^2 = id$, for $2\leq i\leq 14$, and $(\rho_0\beta^{15}\rho_0\beta^{15})^s = id$, where $\beta = \rho_1\rho_2\rho_3\rho_4$.
\end{corollary}

\begin{proposition}
 Let $2\leq s < \infty$. The incidence system $$\mathcal{H}({2 }^{\{5,3,3\},\mathcal{G}(s)}) = \Gamma(H({2 }^{\{5,3,3\},\mathcal{G}(s)}), (H_i)_{i\in \{0,\ldots,4\}}),$$ where $H({2 }^{\{5,3,3\},\mathcal{G}(s)}):=\langle \rho_0\rho_1\rho_0, \rho_1, \rho_2, \rho_3, \rho_4\rangle = \langle \tilde{\rho_0}, \rho_1, \rho_2, \rho_3, \rho_4 \rangle$, is a regular hypertope and its automorphism group, of size $7200 \cdot (2s)^{300}$, has the relations given by its Coxeter diagram
$$
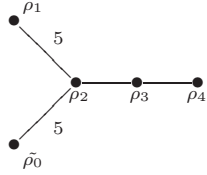
\xymatrix@-1.5pc{*{\bullet}\ar@{-}[rrdd]^(0.01){\rho_1}^5 \\
\\
 &&*{\bullet}\ar@{-}[rr]_(0.05){\rho_2}&&*{\bullet}\ar@{-}[rr]_(0.05){\rho_3}_(0.99){\rho_4}&&*{\bullet}\\
 \\
 *{\bullet}\ar@{-}[rruu]_(0.01){\tilde{\rho_0}}_5\\
 } $$
 and the extra relations $(\tilde{\beta}^{-i}\beta^i)^2 = id$, for $2\leq i\leq 14$, and $(\tilde{\beta}^{15}\beta^{15})^s = id$, where $\beta := \rho_1\rho_2\rho_3\rho_4$ and $\tilde{\beta} := \tilde{\rho_0}\rho_2\rho_3\rho_4$.
\end{proposition}

\section{Final Remarks}\label{futurework}

This work gives a list of infinite families of regular hypertopes of arbitrary rank, constructed from finite centrally symmetric non-degenerate spherical polytopes $\mathcal{P}$. 
We notice that Theorem~\ref{2KG(s)} establishes the following: 
\begin{itemize}
   \item If $s$ is even and $\mathcal{P}$ has only finitely many vertices, then $2^{\mathcal{P},\mathcal{G}(s)}$ is also centrally symmetric.
\end{itemize}
This implies that if $s$ is even, all the polytopes determined in Sections~\ref{rank2&n} and~\ref{rank3&4} are centrally symmetric, being eligible to be used in Theorem~\ref{2KG(s)}, giving other families of polytopes of type $\{4,4,p_1,\ldots,p_{n-1}\}$. Moreover, since we know that these polytopes $2^{\mathcal{P},\mathcal{G}(s)}$ are also non-degenerate, these new families would also be eligible for the halving operation, giving families of hyperpotes.

Here the focus was on spherical polytopes but the same idea can be applied to centrally symmetric toroidal polytopes. Indeed, from Corollary~\ref{cor:2p_polytope} and the point above, it follows that the toroidal maps $\{4,4\}_{(2s,0)}$ are centrally symmetric and non-degenerate (for $s\geq 4$ and even), and therefore extendable by the same processed that was used in this paper. Moreover, the toroidal map $\{4,4\}_{(2s,0)}$ is just a case of the toroidal $(n+1)$-cubic tesselation $\{4,3^{n-2},4\}_{(2s,0^{n-1})}$ for $n=2$.
We can repeat the process to this more general case and extend further this cubic tesselation.

\section{Acknowledgements}

The author would like to thank Asia Ivi{\'c} Weiss for the suggestion to work on finding families of hypertopes from hyperbolic polytopes on the 9th Slovenian International Conference on Graph Theory - Bled 2019, which lead to a much more general result as this one.
The author would also like to thank Maria Elisa Fernandes and Filipe Gomes for the useful comments and suggestions, that greatly improved the manuscript.

This work is supported by The Center for Research and Development
in Mathematics and Applications (CIDMA) through the Portuguese
Foundation for Science and Technology
(FCT - Funda\c{c}\~{a}o para a Ci\^{e}ncia e a Tecnologia),
references UIDB/04106/2020 and UIDP/04106/2020.

\bibliographystyle{acm}
\bibliography{Hypertopes}

\begin{thebibliography}{10}

\bibitem{catalanohypertopes}
{\sc Catalano, D., Fernandes, M.~E., Hubard, I., and Leemans, D.}
\newblock Hypertopes with tetrahedral diagram.
\newblock {\em The electronic journal of combinatorics 25}, 3 (2018), 3--22.

\bibitem{danzer_regular_1984}
{\sc Danzer, L.}
\newblock Regular {Incidence}-{Complexes} and {Dimensionally} {Unbounded}
  {Sequences} of {Such}, {I}.
\newblock In {\em North-{Holland} {Mathematics} {Studies}}, vol.~87. Elsevier,
  1984, pp.~115--127.

\bibitem{rank4toroidal}
{\sc Ens, E.}
\newblock Rank 4 toroidal hypertopes.
\newblock {\em Ars Mathematica Contemporanea 15}, 1 (2018), 67--79.

\bibitem{FLPW2019}
{\sc Fernandes, M.~E., Leemans, D., Piedade, C.~A., and Weiss, A.~I.}
\newblock Two families of locally toroidal regular 4-hypertopes arising from
  toroids.
\newblock {\em Contemp. Math., to appear\/} (2019).

\bibitem{hexagonal}
{\sc Fernandes, M.~E., Leemans, D., and Weiss, A.~I.}
\newblock Hexagonal extensions of toroidal maps and hypermaps.
\newblock In {\em Geometry and Symmetry Conference\/} (2015), Springer,
  pp.~147--170.

\bibitem{HSH}
{\sc Fernandes, M.~E., Leemans, D., and Weiss, A.~I.}
\newblock Highly symmetric hypertopes.
\newblock {\em Aequationes mathematicae 90}, 5 (2016), 1045--1067.

\bibitem{spherical}
{\sc Fernandes, M.~E., Leemans, D., and Weiss, A.~I.}
\newblock An {Exploration} of {Locally} {Spherical} {Regular} {Hypertopes}.
\newblock {\em Discrete \& Computational Geometry 64}, 2 (Sept. 2020),
  519--534.

\bibitem{GAP}
{\sc The GAP~Group}.
\newblock {\em {GAP -- Groups, Algorithms, and Programming, Version 4.11.1}},
  2021.

\bibitem{hartley_new_2008}
{\sc Hartley, M.~I., and Leemans, D.}
\newblock A new {Petrie}-like construction for abstract polytopes.
\newblock {\em Journal of Combinatorial Theory, Series A 115}, 6 (Aug. 2008),
  997--1007.

\bibitem{exitchamb}
{\sc Hou, D.-D., Feng, Y.-Q., and Leemans, D.}
\newblock Existence of regular 3-hypertopes with 2n chambers.
\newblock {\em Discrete mathematics 342}, 6 (2019), 1857--1863.

\bibitem{ARP}
{\sc McMullen, P., and Schulte, E.}
\newblock {\em Abstract regular polytopes}, vol.~92.
\newblock Cambridge University Press, 2002.

\bibitem{teroasia2021}
{\sc Montero, A., and Weiss, A.~I.}
\newblock Proper locally spherical hypertopes of hyperbolic type, preprint,
  arxiv:2102.01157v1.

\bibitem{teroasia2020}
{\sc Montero, A., and Weiss, A.~I.}
\newblock Locally spherical hypertopes from generlised cubes.
\newblock {\em The Art of Discrete and Applied Mathematics\/} (Aug. 2020).

\bibitem{pellicer_extensions_2009}
{\sc Pellicer, D.}
\newblock Extensions of regular polytopes with preassigned {Schl{\"a}fli}
  symbol.
\newblock {\em Journal of Combinatorial Theory, Series A 116}, 2 (Feb. 2009),
  303--313.

\bibitem{schulte_regular_1985}
{\sc Schulte, E.}
\newblock Regular incidence-polytopes with {Euclidean} or toroidal faces and
  vertex-figures.
\newblock {\em Journal of Combinatorial Theory, Series A 40}, 2 (Nov. 1985),
  305--330.

\bibitem{Tits1961}
{\sc Tits, J.}
\newblock Groupes et g{\'e}om{\'e}tries de coxeter, notes polycopi{\'e}es.
\newblock {\em Institut des Hautes {\'E}tudes Scientifiques, Paris\/} (1961).

\end{thebibliography}

\end{document}